\documentclass[11pt]{article}
\usepackage{amsmath,amsthm,amsfonts,array}
\DeclareMathOperator{\1}{id}
\DeclareMathOperator{\Ker}{Ker}

\renewcommand{\=}{\doteq}

\newcommand{\X}{\mathfrak{X}}
\newcommand{\T}{\mathfrak{T}}
\newcommand{\E}{\mathfrak{E}}
\newtheorem{thm}{Theorem}[section]
 \newtheorem{cor}[thm]{Corollary}
 \newtheorem{prop}[thm]{Proposition}
 \newtheorem{lemma}[thm]{Lemma}
\theoremstyle{definition}
 \newtheorem{defn}[thm]{Definition}
\theoremstyle{definition}
\newtheorem{exam}[thm]{Example}
\theoremstyle{definition}
 \newtheorem{rem}[thm]{Remark}
\numberwithin{equation}{section}
\numberwithin{equation}{section}
\begin{document}
\title{\bf  Moufang symmetry VII.\\
Moufang transformations}
\author{Eugen Paal}
\date{}
\maketitle
\thispagestyle{empty}
\begin{abstract}
Concept of a birepresentation for the Moufang loops is elaborated.
\end{abstract}

\section{Introduction}

Groups are often said to be an algebraic abstraction of the notion of symmetry. As a slight generalization of this one can introduce the notion of the Moufang symmetry. The latter can be defined as a hypothetic kind of symmetry associated with the Moufang loops. By paraphrasing these words, we can also say that the Moufang loops are an algebraic abstraction of the Moufang symmetry.

By introducing such a notion of symmetry, one finds oneself confronted with a question about its real meaning. If one feels like looking at world affairs from viewpoint of the Moufang symmetry, one needs  a suitable mathematical machinery for identification of this symmetry. As in case of groups, one really has to elabotate representation theory of the Moufang loops, and this is the logical way to get an answer to the question.

In the present paper we elaborate a concept of a \emph{birepresentation} for the Moufang loops. Throughout the paper,  ideas presented in \cite{Doro} are very useful. 

\section{Moufang loops}

A \emph{Moufang loop} \cite{RM} (see also \cite{Bruck,Bel,HP}) is a set $G$ with a binary operation  (multiplication) 
$\cdot: G\times G\to G$, denoted also by juxtaposition, so that the following three axioms are satisfied:
\begin{enumerate}
\itemsep-3pt
\item[1)] 
in equation $gh=k$, the knowledge of any two of $g,h,k\in G$ specifies the third one \emph{uniquely},
\item[2)] 
there is a distinguished element $e\in G$ with the property $eg=ge=g$ for all $g\in G$,
\item[3)] 
the \emph{Moufang identity} 
\begin{equation}
\label{Moufang}
(gh)(kg) = g(hk)g
\end{equation}
hold in $G$.
\end{enumerate}
Recall that a set with a binary operation is called a \emph{groupoid}. A groupoid $G$ with axiom 1) is called a \emph{quasigroup}. If axioms 1) and 2) are satisfied, the grupoid (quasigroup) $G$ is called a \emph{loop}. The element $e$ in axiom 2) is called the \emph{unit} (element) of the (Moufang) loop $G$.

In a (Moufang) loop, multiplication need not be neither associative nor commutative. Associative (Moufang) loops are well known and called \emph{groups}. The \emph{associativity} and \emph{commutativity} laws read, respectively,
\begin{equation*}
g\cdot hk=gh\cdot k,\quad gh=hg, \qquad \forall g,h,k\in G
\end{equation*}
The most familiar kind of loops are those with the \emph{associative} law, and these are called 
\emph{groups}. A (Moufang) loop $G$ is called \emph{commutative} if the commutativity law holds in $G$, and (only) the commutative associative (Moufang) loops are said to be \emph{Abelian}.

The most remarkable property of the Moufang loops is their \emph{diassociativity}: in a Moufang loop $G$ every two elements generate an associative subloop (group) \cite{RM}. In particular, from this it follows that
\begin{equation} 
\label{alt}
g\cdot gh=g^{2}h,\quad 
hg\cdot g=hg^{2},\quad 
gh\cdot g=g\cdot hg,\qquad \forall g,h\in G
\end{equation}
The first and second identities in (\ref{alt}) are called the left and right \emph{alternativity}, respectively, and the third one is said to be \emph{flexibility}. 

The unique solution of equation $xg=e$ ($gx=e$) is called the left (right) \emph{inverse} element of $g\in G$ and is denoted as $g^{-1}_{R}$ ($g^{-1}_{L}$). It follows from diassociativity of the Moufang loop that 
\begin{subequations}
\label{inverses}
\begin{gather}
g^{-1}_{R}=g^{-1}_{L}\doteq  g^{-1}\\
g^{-1}\cdot gh=hg\cdot g^{-1}\\
\left(g^{-1}\right)^{-1}=g\\
(gh)^{-1}=h^{-1}g^{-1},\qquad \forall g,h\in G
\end{gather}
\end{subequations}

\section{Moufang transformations}

Let $\X$ be a set and let $\T(\X)$ denote the transformation group of $X$. Elements of $\X$ are called \emph{transformations} of $\X$. Multiplication in $\T(\X)$ is defined as composition of transformations, and  unit element of $\T(\X)$ coincides with the  identity transformation $\1$ of $X$.

Let $G$ be a Moufang loop with the unit element $e\in G$ and let $(S,T)$ denote a pair of maps $S,T:G\to\T(\X)$. 
\begin{defn}[birepresentation]
The pair $(S,T)$ is said to be an \emph{action} of $G$ on $\X$ if
\begin{subequations}
\label{bir_def}
\begin{gather}
S_e=T_e=\1\\
S_gT_gS_h=S_{gh}T_g\\
S_gT_gT_h=T_{hg}S_g
\end{gather}
hold for all $g,h$ in $G$.
\end{subequations}
The pair $(S,T)$ is called also a \emph{birepresentation} of $G$ (in $\T(\X)$). Transformations $S_g,T_g\in\T(\X)$ ($g\in G$) are called $G$-transformations or the \emph{Moufang transformations} of $\X$. The set of all Moufang transformations is denoted as $\E_G(S,T)$.
\end{defn}

\begin{exam}
Define the left ($L$) and right ($R$) translations of $G$ by $gh=L_gh=R_hg$.
Then it follows from the (\ref{Moufang}) that the pair $(L,R)$ of maps $L_g,R_g:G\to \T(G)$ is a birepresentation of $G$ in $\T(G)$.
\end{exam}

The Moufang transformations need not close, but generate a subgroup of $\T(\X)$. This subgroup is called an 
\emph{enveloping group} of $\E_G(S,T)$ and is denoted as $\overline{\E_G(S,T)}$. In other words, the Moufang transformations are generators of group $\overline{\E_G(S,T)}$ -- the enveloping group of birepresentation $(S,T)$ of $G$. The defining relations of $\overline{\E_G(S,T)}$ are (\ref{bir_def}a--c). The enveloping group $\overline{\E_G(S,T)}$  can be called the \emph{multiplication group} of the birepresentation $(S,T)$ as well.

\begin{defn}[kernel]
The set
\begin{equation*}
K\=\Ker(S,T)\=\{g\in G|\, S_g=T_g=\1\}
\end{equation*}
is called the \emph{kernel} of birepresentation $(S,T)$. If $K=\{e\}$, then birepresentation $(S,T)$ is called  \emph{faithful} and action of the Moufang loop $G$ on $\X$ is called \emph{effective}.
\end{defn}

\begin{exam}
Birepresentation $(L,R)$ is exact.
\end{exam}

\section{Properties of Moufang transformations}

\begin{prop}
We have
\begin{equation}
\label{ST_comm}
S_g T_g=T_g S_g,\quad \forall g\in G
\end{equation}
\end{prop}

\begin{proof}
Set $h=u$ in (\ref{bir_def}c)
\end{proof}

\begin{prop}
We have
\begin{equation}
\label{ST_inverses}
S_{g^{-1}}=S^{-1}_g,\quad T_{g^{-1}}=T^{-1}_g, \quad \forall g\in G
\end{equation}
\end{prop}

\begin{proof}
In (\ref{bir_def}b,c) first set $h=g^{-1}$:
\begin{equation*}
S_gS_{g^{-1}}=T_gT_{g^{-1}}=\1
\end{equation*}
Analogously, setting in (\ref{bir_def}b,c) $g=h^{-1}$, we have
\begin{equation*}
S_{h^{-1}}S_h=T_{h^{-1}}T_h=\1
\end{equation*}
Thus
\begin{equation*}
S_gS_{g^{-1}}=S_{g^{-1}}S_g=\1,\quad
T_gT_{g^{-1}}=T_{g^{-1}}T_g=\1,\quad \forall g\in G
\tag*{\qed}
\end{equation*}
\renewcommand{\qed}{}
\end{proof}

\begin{lemma}
The defining relations of the Moufang transformations can equivalently be written as follows:
\begin{subequations}
\label{bir_def2}
\begin{gather}
S_e=T_e=\1\\
S_hT_gS_g=T_gS_{hg}\\
T_hT_gS_g=S_gT_{gh}
\end{gather}
\end{subequations}
for all $g,h$ in $G$.
\end{lemma}

\begin{proof}
It follows from (\ref{bir_def}b) and (\ref{ST_inverses}) that
\begin{equation*}
S^{-1}_h T^{-1}_g S^{-1}_g 
= T^{-1}_g S^{-1}_{gh}
\end{equation*}
which implies
\begin{equation*}
S_{h^{-1}} T_{g^{-1}} S_{g^{-1}} 
= T_{g^{-1}} S_{(gh)^{-1}}
= T_{g^{-1}} S_{h^{-1}g^{-1}}
\end{equation*}
Thus, replacing $g^{-1}\to g$ and $h^{-1}\to h$  we obtain (\ref{bir_def2}b). Analogously (\ref{bir_def2}c) can be checked.
\end{proof}

\begin{thm}
The Moufang transformations satisfy the following relation:
\begin{equation}
\label{ST_comm2}
S_gS_hT_hT_g=T_hT_gS_gS_h,\quad \forall g,h\in G
\end{equation}
\end{thm}

\begin{proof}
In (\ref{bir_def}b) interchange $g$ and $h$ to obtain 
\begin{equation*}
S_hT_gS_g=T_gS_{hg}
\end{equation*}
and comparing the resulting formula with (\ref{bir_def}b) we get
\begin{equation*}
S_g T_g S_h T^{-1}_g 
=T^{-1}_h S_g T_ h S_h
\end{equation*}
which implies the the desired relation.
\end{proof}

\begin{lemma}
\label{bir_def3}
The defining relations of the Moufang transformations satisfy the following relations:
\begin{align*}
S_{g^{-1}h}&=T^{-1}_g S^{-1}_g S_h T_g\\
T_{g^{-1}h}&=S_g T_h  T^{-1}_g S^{-1}_g\\
S_{hg^{-1}}&=T^{-1}_g S^{-1}_y S^{-1}_g T^{-1}_g\\
S_{hg^{-1}}&=S^{-1}_g T^{-1}_g  T_h S_g
\end{align*}
for all $g,h$ in $G$.
\end{lemma}

\begin{thm}
\label{ST_properties}
We have:
\begin{enumerate}
\itemsep-3pt
\item[1)]
$\Ker(S,T)$ is a subloop of the Moufang loop $G$,
\item[2)]
$S_g=S_h$ ($T_g=T_h$) iff and only if $S_{g^{-1}h}=\1$ ($T_{g^{-1}h}=\1$),
\item[3)]
birepresentation $(S,T)$ is faithful if and only if from $S_g=S_h$ and $T_g=T_h$ follows $g=h$.
\end{enumerate}
\end{thm}

\begin{proof}
Use Lemma \ref{bir_def2}.
\end{proof}

\section{Triality}

Define the \emph{quadratic} Moufang transformations  as
\begin{equation}
\label{P_def}
P_g\=S^{-1}_g T^{-1}_g\quad \in \overline{\E_G(S,T)},\quad g\in G
\end{equation}
Note that $P_g$ commutes both with $S_g$ and $T_g$. Thus we can equivalently define $P_g$ by the symmetric relation
\begin{equation}
\label{STP_def}
S_gT_gP_g\=\1,\quad g\in G
\end{equation}
\begin{prop}
We have
\begin{gather*}
P_e=\1\\
P^{-1}_g P_g =P_g P^{-1}_g = \1,\quad \forall g\in G
\end{gather*}
\end{prop}

\begin{cor}
We have
\begin{equation*}
P^{-1}_g=P_{g^{-1}},\quad \forall g\in G
\end{equation*}
\end{cor}

Denote by $(S,T,P)$ the triple of maps $S,T,P: g\to \T(\X)$.

\begin{thm}
Let $(S,T)$ be a birepresentation of the Moufang loop $G$. Then the following pairs are birepresentations of $G$ as well:
\begin{align*}
(T^{-1},S^{-1})&: g\to T^{-1}_g,\quad S^{-1}_g\\
(T,P)&: g\to T_g,\quad g\to P_g\\
(P^{-1},P^{-1})&: g\to P^{-1}_g,\quad T^{-1}_g\\
(P,S)&: g\to P_g,\quad g\to S_g\\
(S^{-1},P^{-1})&: g\to S^{-1}_g,\quad P^{-1}_g
\end{align*}
\end{thm}

\begin{proof}
As an example, check the definig relations for pair $(T,P)$:
\begin{align*}
T_gP_gT_h=T_{gh}P_g\\
T_gP_gP_{gh}=P_{hg}T_g
\end{align*}
To get the first relation, express $P_g$ from (\ref{STP_def}) and replace into this relations, the resulting relation is equivalent to (\ref{bir_def}b). The second relation can equivalently be written as
\begin{equation*}
P_{hg^{-1}}=S_gP_hT_g
\end{equation*}
Now calculate:
\begin{align*}
P_{hg^{-1}}
\overset{(\ref{P_def})}{}
&=S^{-1}_{hg^{-1}}T^{-1}_{hg^{-1}}\\
\overset{(\ref{ST_inverses}), (\ref{inverses}d)}{}
&=S_{gh^{-1}}T_{gh^{-1}}\\
\overset{(\ref{bir_def}b), (\ref{bir_def2}c), (\ref{ST_inverses})}{}
&=S_gT_gS^{-1}_hT^{-1}_gS^{-1}_gT^{-1}_hT_gS_g\\
\overset{(\ref{ST_comm})}{}
&=S_gT_gS^{-1}_hS^{-1}_gT^{-1}_gT^{-1}_hT_gS_g\\
\overset{(\ref{ST_comm2})}{}
&=S_gT_gT^{-1}_gT^{-1}_hS^{-1}_hS^{-1}_gT_gS_g\\
\overset{(\ref{P_def})}{}
&=S_gT_hT_g
\end{align*}
The defining relations for other pairs can be checked analogously.
\end{proof}

\begin{cor}
The defining relations of the birepresentatitons from triple $(S,T,P)$ can be collected to the following table: 
\begin{center}
\begin{tabular}{|c|c|c|c|c|c|}
\hline
$(S,T)$&$(T^{-1},S^{-1})$&$(T,P)$&$(P^{-1},T^{-1})$&$(P,S)$&$(S^{-1},P^{-1})$\\
\hline\hline 
{\rm(\ref{bir_def4}a)}&{\rm(\ref{bir_def4}b)}&{\rm(\ref{bir_def4}b)}&
{\rm(\ref{bir_def4}c)}&{\rm(\ref{bir_def4}c)}&{\rm(\ref{bir_def4}a)}\\
\hline 
{\rm(\ref{bir_def5}b)}&{\rm(\ref{bir_def5}a)}&{\rm(\ref{bir_def5}c)}&
{\rm(\ref{bir_def5}b)}&{\rm(\ref{bir_def5}a)}&{\rm(\ref{bir_def5}c)}\\
\hline
\end{tabular}
\end{center}
where
\begin{gather}
\label{bir_def4}
S_{g^{-1}h}\overset{(a)}{=}P_xS_hT_g,\quad
T_{g^{-1}h}\overset{(b)}{=}S_xT_hP_g,\quad
P_{g^{-1}h}\overset{(c)}{=}T_xP_hS_g\\
\label{bir_def5}
S_{hg^{-1}}\overset{(a)}{=}T_gS_hP_g,\quad
T_{hg^{-1}}\overset{(b)}{=}P_gT_hS_g,\quad
P_{hg^{-1}}\overset{(c)}{=}S_gP_hT_g
\end{gather}
\end{cor}

\begin{cor}
It follows from (\ref{bir_def4}a--c) and  (\ref{bir_def5}a--c) that
\begin{gather*}
P_xS^{-1}_hT_g=T_gS^{-1}_hP_g,\quad
S_xT^{-1}_hP_g=P_gT^{-1}_hS_g,\quad
T_xP^{-1}_hS_g=S_gP^{-1}_hT_g
\end{gather*}
The latter are equivalent to (\ref{ST_comm}) and to
\begin{gather*}
T_gT_hP_hP_g=P_hP_gT_gT_h,\quad
T_gT_hP_hP_g=P_hP_gT_gT_h 
\end{gather*}
\end{cor}

Collecting above properties of birepresentations we can propose

\begin{thm}[principle of triality]
The definig relations of the Moufang transformations are invariant under the triality substitutions
\begin{align*}
\1&=(S\to S)(T\to T)(P\to P)\\
\tau
&=(S\to  T^{-1}\to S)(P\to P^{-1})\\
\rho
&=(S\to T\to P\to S)\\
\rho^2
&=(S\to P\to T\to S)\\
\rho\circ\tau
&=(S\to P^{-1}\to S)(T\to T^{-1})\\
\rho^2\circ\tau
&=(T\to  P^{-1}\to P)(S\to S^{-1})
\end{align*}
Hence all algebraic consequences of the defining relations must be triality invariant as well.
\end{thm}

\section{Reconstruction Theorem}

It turns out that the triality symmetry is a characteristic property of the Moufang transformations.

\begin{thm}[reconstruction]
Let $G$ be a groupoid and $(S,TP)$ a triple of maps $S,T,P:G\to \T(\X)$ such that:
\begin{itemize}
\itemsep-3pt
\item[1)]
$S_gT_gP_g=\1$ for all $g$ in G,
\item[2)]
for every $g$ in $G$ there exists $\overline{g}$ in $G$ such that $S^{-1}_x=S_{\overline{g}}$ and
$T^{-1}_x=T_{\overline{g}}$,
\item[3)]
for all $g,h$ in $G$ relations
\begin{gather*}
S_{\overline{g}h}=P_xS_hT_g,\quad
T_{\overline{g}h}=S_xT_hP_g,\quad
P_{\overline{g}h}=T_xP_hS_g\\
S_{h\overline{g}}=T_gS_hP_g,\quad
T_{h\overline{g}}=P_gT_hS_g,\quad
P_{h\overline{g}}=S_gP_hT_g
\end{gather*}
are satisafied in $\T(\X)$,
\item[4)]
from $S_g=S_h$ and $T_g=T_h$  it follows that $g=h$.
\end{itemize}
Then $G$ is a Moufang loop. The unit element of $G$ is $g\overline{g}=\overline{g}g\=e$, where the latter does not depend on the choice of $g$ in $G$, and the inverse element of $g$ is $\overline{g}$.
\end{thm}

\begin{proof}
The detailed proof is presented in \cite{Paal8}
\end{proof}

\section{Triple closure}

\begin{thm}
\label{triple-closure_thm}
The Moufang transformations satisfy the triple closure relations:
\label{triple_closure}
\begin{gather}
S_gS_hSg\overset{(a)}{=}S_{ghg},\quad
T_gT_hTg\overset{(b)}{=}T_{ghg},\quad
P_gP_hPg\overset{(c)}{=}P_{ghg},\quad \forall g\in G
\end{gather}
\end{thm}

\begin{proof}
Calculate:
\begin{align*}
S_{ghg}
&\overset{(\ref{bir_def4})}{=}P_{g^{-1}}S_{hg}T_{g^{-1}}\\
&\overset{(\ref{bir_def5})}{=}P_{g^{-1}}T_{g^{-1}}S_hP_{g^{-1}} T_{g^{-1}}\\
&\overset{(\ref{STP_def})}{=}S^{-1}_{g^{-1}}S_hS^{-1}_{g^{-1}}\\
&\overset{(\ref{ST_inverses})}{=}S_gS_hSg
\end{align*}
The remaining relations (\ref{triple_closure}b,c) can be checked analogously.
\end{proof}

\begin{rem}
It follows from Theorem (\ref{triple-closure_thm}) that the Moufang transformations realize the 
\emph{triple family of transformations} \cite{Nono} of $\X$
\end{rem}

\section{Minimality  conditions}

We call birepresentation $(S,T)$ \emph{associtive} if the Moufang transformations satisfy the closure relations
\begin{equation}
\label{assoc1}
S_gS_h\overset{(a)}{=}S_{gh},\quad
T_gT_h\overset{(b)}{=}T_{gh},\quad\
S_gT_h\overset{(c)}{=}T_hS_g,\quad \forall g,h\in G
\end{equation}
It follows from (\ref{bir_def}b,c) that these conditions are equivalent.

It has to be noted that the non-associative Moufang loops do not have faithful associative birepresentations.
Really, for the associtive birepresentation we have
\begin{equation*}
S_gS_hS_k=S_{gh\cdot k}=S_{g\cdot h k},\quad
T_gT_hT_k=S_{gh\cdot k}=T_{g\cdot h k}
\end{equation*}
from which it follows that $(gh\cdot k)^{-1}(g\cdot h k)\in\Ker(S,T)$. But for the faithful birepresentation $\Ker(S,T)=\{e\}$, hence $g\cdot h k=g\cdot h k$.

Denote the commutator of transformations $A,B$ by $[A,B]\=ABA^{-1}B^{-1}$.
Equivalence of the associtivity constraints (\ref{assoc1}) can be also seen from 
\begin{thm}[minimality conditions]
The Moufang transformations satisfy relations
\begin{equation}
\label{mini1}
[T_h,S^{-1}_g]
\overset{(a)}{=}S^{-1}_{gh}S_gS_h
\overset{(b)}{=}T_{gh}T^{-1}_gT^{-1}_h
\overset{(c)}{=}[S^{-1}_g,T_h]
\overset{(d)}{=}S^{-1}_gS^{-1}_hS_{hg}
\overset{(e)}{=}T_gT_hT^{-1}_{hg}
\end{equation}
\end{thm}

\begin{proof}
It is easy to check that 
(\ref{mini1}a)$\equiv$(\ref{bir_def4}a),
(\ref{mini1}b)$\equiv$(\ref{bir_def4}c),
(\ref{mini1}c)$\equiv$(\ref{bir_def4}b),
(\ref{mini1}e)$\equiv$(\ref{bir_def4}c) 
and
(\ref{mini1}a)$\equiv$(\ref{bir_def4}a)
Note that other possible equalities from  (\ref{mini1}a--e) give rise also (\ref{bir_def4}a--c), (\ref{bir_def5}a--c) or the triple closure relations (\ref{triple_closure}a--c).
\end{proof}

\begin{defn}[associators]
Let $(S,T)$ be a birepresentation of the Moufang loop $G$. Elements from group $\overline{\E_G(S,T)}$ of form
\begin{align*}
S(g;h)&\=S^{-1}_{gh}S_gS_h\\
T(g;h)&\=T_{gh}T^{-1}_gT^{-1}_h\\
[T_g,S^{-1}_h]&\=T_gS^{-1}_hT^{-1}_gS_h\\
[S^{-1}_g,T_h]&\=S^{-1}_gT_hS_gT^{-1}_h
\end{align*}
are called \emph{associators} of birepresentation $(S,T)$.
\end{defn}

It is easy to see from Theorem \ref{mini1}:
\begin{cor}[minimality conditions]
Associator of a birepresentation $(S,T)$ satisfy the minimality conditions
\begin{equation}
\label{mini2}
[T_g,S^{-1}_h]
=S(g;h)
=T(g;h)
=[S^{-1}_g,T_h]
=S^{-1}(h;g)
=T^{-1}(h;g)
\end{equation}
\end{cor} 

\begin{rem}
For associative Moufang transformations we have
\begin{equation}
\label{assoc2}
[T_g,S^{-1}_h]
=S(g;h)
=T(g;h)
=[S^{-1}_g,T_h]
=S^{-1}(h;g)
=T^{-1}(h;g)=\1
\end{equation}
Comparing (\ref{mini2}) and  (\ref{assoc2}) one can say that the Moufang transformations have the property that their associativity is spoiled in the \emph{minimal} way. Constraints  (\ref{mini1}) and  (\ref{mini2}) are hence called the \emph{minimality conditions}
\end{rem}

By triality we can propose
\begin{thm}[triality and minimality]
The Moufang transformations satisfy the minimality conditions:
\begin{gather}
\label{mini3}
[P_h,T^{-1}_g]
\overset{(a)}{=}T^{-1}_{gh}T_gT_h
\overset{(b)}{=}P_{gh}P^{-1}_gP^{-1}_h
\overset{(c)}{=}[T^{-1}_g,T_h]
\overset{(d)}{=}T^{-1}_gT^{-1}_hT_{hg}
\overset{(e)}{=}P_gP_hP^{-1}_{hg}\\
\label{mini4}
[S_h,P^{-1}_g]
\overset{(a)}{=}P^{-1}_{gh}P_gP_h
\overset{(b)}{=}S_{gh}S^{-1}_gS^{-1}_h
\overset{(c)}{=}[P^{-1}_g,P_h]
\overset{(d)}{=}P^{-1}_gP^{-1}_hP_{hg}
\overset{(e)}{=}S_gS_hS^{-1}_{hg}
\end{gather}
\end{thm}

\begin{proof}
Constraints (\ref{mini3}a--e) and  (\ref{mini4}a--e) hold because $(T,P)$ and $(P,S)$ are birepresentations of $G$.
\end{proof}

\section{Theorem on kernel of birepresentation}

\begin{defn}[normal divisor \cite{Bruck}] 
A subloop $N$ of the Moufang loop $G$ is called a \emph{normal divisor} of $G$ if it is invariant with respect to the following trensformations of $\X$ from the group $\overline{\E_G(S,T)}$:
\begin{equation*}
L(g;h)\=L^{-1}_{gh}L_gL_h,\quad
M^+_g\=R_gL^{-1}_g
\end{equation*}
If $L(g;h)=\1$ for all $g,h$ in $G$, the Moufang loop $G$ is a group and then every $M^+_g$ ($g\in G$) is an \emph{innner automorphism} of $G$ 
\end{defn}

\begin{thm}
The kernel $\Ker(S,T)$ of a birepresentation $(S,T)$ of the Moufang loop $G$ is a normal divisor of $G$.
\end{thm}

\begin{proof}
We know from Theorem \ref{ST_properties} $\Ker(S,T)$ is a subloop of $G$, thus it is  sufficient to check that for all $g,h$ in $G$ and $k$ in $\Ker(S,T$ we have
\begin{gather}
\label{ker1}
S_{M^{+}_gk}\overset{(a)}{=}\1,\quad 
T_{M^{+}_gk}\overset{(a)}{=}\1\\
\label{ker2}
S_{L(g,h)k}\overset{(a)}{=}\1,\quad 
T_{L(g,h)k}\overset{(a)}{=}\1
\end{gather}
First calculate
\begin{align*}
S_{M^{+}_gk}
&=S_{R_gL^{-1}_gk}\\
&=T^{-1}_gS_{L^{-1}_gk}P^{-1}_g\\
&=T^{-1}_gP_gS_kT_gP^{-1}_g\\
&=T^{-1}_gP_g\1 T_gP^{-1}_g\\
&=T^{-1}_gP_g\,T_gP^{-1}_g\\
&=T^{-1}_gS^{-1}_gP^{-1}_g\\
&=(P_gS_gT_g)^{-1}\\
&=\1
\end{align*}
Condition (\ref{ker1}b) can be checked analogously. Next calculate
\begin{align*}
S_{L(g;h)k}
&=S_{L^{-1}_{gh}L_gL_hk}\\
&=P_{gh}S_{L_gL_hk}T_{gh}\\
&=P_{gh}P^{-1}_gS_{L_hk}T^{-1}_gT_{gh}\\
&=P_{gh}P^{-1}_gP^{-1}_hS_{k}T^{-1}_hT^{-1}_gT_{gh}\\
&=P_{gh}P^{-1}_gP^{-1}_h\1 T^{-1}_hT^{-1}_gT_{gh}\\
&=P_{gh}P^{-1}_gP^{-1}_h\,T^{-1}_hT^{-1}_gT_{gh}\\
\overset{(\ref{mini3}b)}{}&=\1
\end{align*}
Condition (\ref{ker2}b) can be checked analogously.
\end{proof}

\section{Birepresentation of quotient loop $G/\Ker(S,T)$}

Recall some basic facts \cite{Bruck} from theory of the Moufang loops.

Let $N$ be a normal divisor of $(S,T)$. The we can define on the Moufang loop $G$ the \emph{left} (\emph{right}) \emph{equivalence}: for $g,h$ in $G$ we set $g\overset{L}{\sim}h$ ($g\overset{R}{\sim}h$) if $g^{-1}h\in N$ ($hg^{-1}\in N$).  The resulting equivalence classes are called the \emph{left} (\emph{right}) \emph{cosets} with respect to the normal divisor $N$. It turns out \cite{Bruck} that the left and right cosets can be presented as $gN$ and $Ng$, respectively, and coincide: $gN=Ng$. On the set of cosets of $G$ with respect to $N$  we can define multiplication:
\begin{equation*} 
(gN)(hN)\=(gh)N
\end{equation*}
which satisfy all the Moufang loop axioms. The resulting Moufang loop is called the quotient loop with respect to $N$ and is denoted by $G/N$. The unit element of $G/N$ is $K$.

The normal divisors coincide \cite{Bruck} with kernels of homomorphisms.

\begin{thm}
Let $(S,T)$ be a birepresentation of the Moufang loop $G$ and $K\=\Ker(S,T)$ be kernel of the birepresentation $(S,T)$. Then the pair of maps $gK\to S_g$, $gk\to T_g$ is a faithful birepresentation of the quotient Moufang loop $G/K$.
\end{thm}

\begin{proof}
A pair $(S',T')$ of maps $gK\to S'_{gK}$, $gk\to T'_{gK}$ is a birepresentation of $G$ if the following conditions are satisfied:
\begin{subequations}
\label{bir_def_quotient}
\begin{gather}
S'_K=T'_K=\1\\
S'_{gk}T'_{gK}S'_{hK}=S'_{(gK)(hK)}T'_{gK}=S'_{(gh)K}T'_{gK}\\
S'_{gk}T'_{gK}T'_{hk}=T'_{(hK)(gK)}S'_{gK}=T'_{(hg)K}S'_{gK}
\end{gather}
\end{subequations}
Define $S'_{gK}$ and $S'_{gK}$ by the following simple formulae:
\begin{equation*}
S'_{gK}\=S_g,\quad 
T'_{gK}\=T_g,\quad \forall gK\in G/K
\end{equation*}
First of all, note that the definition of $S'_{gK}$ and $S'_{gK}$ does depend on the choice of representatives in coset $gK$. Really, if $k\in\ gK$, then $k=gn$, where $n\in K$. Then we have
\begin{gather*}
S'_{kK}=S_k=S_{gn}=P^{-1}_gS_nT^{-1}_g=P^{-1}_gT^{-1}_g=S_g\\
T'_{kK}=T_k=T_{gn}=T^{-1}_gS_nP^{-1}_g=T^{-1}_gP^{-1}_g=T_g
\end{gather*}
Thus the maps $gK\to S'_{gK}$, $gk\to T'_{gK}$ are defined uniquely. The definig relations of (\ref{bir_def_quotient}b,c) follow from (\ref{bir_def}a--c) and the above definition of $S'_{gK}$ and $'S_{gK}$. This means that the pair of maps
$gK\to S'_{gK}$, $gk\to T'_{gK}$ is a birepresentation of $G/K$. The set 
\begin{equation*}
\Ker(S',T')\=\{gK\in G/K|\, S'_{gK}=T'_{gK}=\1\}
\end{equation*}
is the kernel of birepresentation $(S',T')$. Evidently, $K\in\Ker(S',T')$. If $gK\in\Ker(S',T')$, then it follws from
\begin{equation*}
S'_{gK}=S_g=\1,\quad T'_{gK}=T_g=\1
\end{equation*}
that $g\in K$. Thus, $\Ker(S',T')=\{K\}$, from which it follows that birepresentation $gK\to S'_{gK}\=S_g$, $gK\to T'_{gK}\=T_g$ is faithful.
\end{proof}

\section*{Acknowledgement}

Research was in part supported by the Estonian Science Foundation, Grant 6912.

\bigskip\noindent
Department of Mathematics\\
Tallinn University of Technology\\
Ehitajate tee 5, 19086 Tallinn, Estonia\\ 
E-mail: eugen.paal@ttu.ee

\end{document}